\documentclass[12pt,reqno]{amsart}
\usepackage{mathtools}
\usepackage{amssymb}
\usepackage{textgreek}
\usepackage{extarrows}
\usepackage{tikz}
\usetikzlibrary{calc}
\usetikzlibrary{arrows.meta}
\usepackage[normalem]{ulem}

\begin{document}
\numberwithin{equation}{section}

\newcommand{\UH}{\mathbb{H}}
\newcommand{\R}{\mathbb R}
\newcommand{\Z}{\mathbb Z}
\newcommand{\C}{\mathbb C}
\newcommand{\zbar}{\overline{z}}
\newcommand{\B}{\mathbb B}
\newcommand{\Hol}{{\sf Hol}}
\newcommand{\Aut}{{\sf Aut}(\mathbb D)}
\newcommand{\D}{\mathbb D}
\newcommand{\N}{\mathbb N}

\def\Re{\mathop{\mathtt{Re\hskip.2em}}}
\def\Im{\mathop{\mathtt{Im\hskip.2em}}}

\newcommand{\Real}{\mathbb{R}}
\newcommand{\Natural}{\mathbb{N}}
\newcommand{\Complex}{\mathbb{C}}
\newcommand{\ComplexE}{\overline{\mathbb{C}}}
\newcommand{\Int}{\mathbb{Z}}
\newcommand{\UD}{\mathbb{D}}

\newcommand{\Step}[1]{\medskip\noindent{\sc Step~#1.} }
\newcommand{\claim}[2]{\begin{itemize}\item[{\it Claim~#1.}]{\it #2}\end{itemize}}

\newcommand{\mcite}[1]{\csname b@#1\endcsname}
\newcommand{\UC}{{\partial\UD}}

\newcommand{\Moeb}{\mathrm{M\ddot ob}}

\newcommand{\diam}{\mathsf{diam}}

\theoremstyle{theorem}
\newtheorem {result} {Theorem}
\setcounter {result} {64}
 \renewcommand{\theresult}{\char\arabic{result}}

\def\dist{{\mathsf{dist}}}
\def\const{{\rm const}}
\def\id{{\sf id}}

\newcommand{\sgn}{\mathop{\mathrm{sgn}}}

\emergencystretch15pt \frenchspacing

\newtheorem{theorem}{Theorem}
\newtheorem{conjecture}{Conjecture}
\newtheorem{lemma}{Lemma}[section]
\newtheorem{proposition}[lemma]{Proposition}
\newtheorem{corollary}[lemma]{Corollary}

\theoremstyle{definition}
\newtheorem{definition}[lemma]{Definition}
\newtheorem{example}[lemma]{Example}

\theoremstyle{remark}
\newtheorem{remark}[lemma]{Remark}
\numberwithin{equation}{section}

\newcommand{\Maponto}
{\xrightarrow{\hbox{\lower.2ex\hbox{$\scriptstyle \smash{\mathsf{onto}}$}}\,}}
\newcommand{\Mapinto}
{\xrightarrow{\hbox{\lower.2ex\hbox{$\scriptstyle \smash{\mathsf{into}}$}}\,}}
\newcommand{\anglim}{\angle\lim}

\newenvironment{mylist}{\begin{list}{}%
{\labelwidth=2em\leftmargin=\labelwidth\itemsep=.4ex plus.1ex
minus.1ex\topsep=.7ex plus.3ex
minus.2ex}%
\let\itm=\item\def\item[##1]{\itm[{\rm ##1}]}}{\end{list}}

\newcommand{\di}{\,\mathrm{d}}
\newcommand{\A}{\mathbb{A}}
\newcommand{\nlim}{\text{as~$~n\to+\infty$}}
\newcommand{\moverline}[1]{\!\overline{\,#1\!}\,}

\newcommand{\supp}{\mathop{\mathsf{supp}}}
\long\def\REM#1{\relax}

\newcommand{\SHOWCORRECTIONS}{%
\newcommand{\PN}[1]{{\color{blue!80!black}##1}}%
\newcommand{\PD}[1]{{\color{green!20!gray}\sout{##1}}}%
\newcommand{\PC}[1]{{\color[rgb]{0.25,0.5,0.6}##1}}%
\newcommand{\ON}[1]{{\color{red!80!black}##1}}
\newcommand{\OD}[1]{{\color{red!20!gray}\sout{##1}}}%
\newcommand{\OC}[1]{{\color[rgb]{0.6,0.5,0.25}##1}}%
\newcommand{\RD}[1]{\textcolor{red}{##1}}%
}%

\newcommand{\HIDECORRECTIONS}{%
\newcommand{\PN}[1]{##1}
\newcommand{\PD}[1]{\relax}%
\newcommand{\PC}[1]{\relax}%
\let\ON=\PN%
\let\OD=\PD%
\let\OC=\PC%
\newcommand{\RD}[1]{##1}%
}%

\SHOWCORRECTIONS

\author[P. Gumenyuk]{Pavel Gumenyuk}
\address{P. Gumenyuk: Department of Mathematics, Politecnico di Milano, via E. Bonardi 9, 20133 Milan, Italy.}
\email{pavel.gumenyuk@polimi.it}

\author[O. Roth]{Oliver Roth}
\address{O. Roth: Department of Mathematics, University of W\"urzburg, Emil Fischer Strasse~40, 97074, W\"urzburg, Germany.} \email{roth@mathematik.uni-wuerzburg.de}

\title[On the squeezing function for finitely connected planar domains]{On the squeezing function for finitely connected planar domains}

\date{\today}


\let\le=\leqslant
\let\ge=\geqslant

\begin{abstract}
In a recent paper, Ng, Tang and Tsai (Math. Ann. 2020) have found an explicit formula for the squeezing function
of an annulus via the Loewner differential equation. Their result has led them to conjecture a corresponding formula for planar domains of any finite connectivity stating that the extremum in the squeezing function problem is achieved for a suitably chosen conformal mapping onto a circularly slit disk. In this paper we disprove this conjecture. We also give a conceptually simple potential--theoretic proof of the explicit formula for the squeezing function of an annulus which has the added advantage of identifying all extremal functions.
\end{abstract}

\maketitle

\section{Introduction}
\noindent Let  ${\Omega\subset\C^d}$, $d\ge1$, be a domain such that the class $\,\mathcal U(\Omega)$ of all injective holomorphic mappings ${f : \Omega \to \B:=\{z\in\C^d\colon |z_1|^2+\cdots+|z_d|^2<1\}}$ is not empty.
We denote by $\dist(0,\partial f(\Omega))$ the Euclidean distance of the origin $0$ from the boundary of $f(\Omega)$.
The \hbox{\textsl{squeezing function}} $S_{\Omega} : \Omega \to \R$ of the domain $\Omega$ is defined by
\begin{equation}\label{EQ_main-definition}
 S_{\Omega}(z):=\sup\big\{\dist\big(0,\partial f(\Omega)\big)\colon f\in\mathcal U(\Omega), f(z)=0\big\}, \qquad z\in \Omega.
\end{equation}

This notion was introduced in 2012 by Deng, Guang and Zhang \cite{DengGuanZhang2012} inspired by the works of Liu, Sun and Yau \cite{LiuSunYau2004, LiuSunYau2005}~(2004) and Yeung~\cite{Yeung2009}~(2009).
Squeezing functions and their properties have since been investigated by many authors; we refer to the papers [\mcite{Deng2016},\,\mcite{DengGuanZhang2016}\,--\,\mcite{FornaessWold2018}, \mcite{JooKimm2018}, \mcite{KimZhang2016}, \mcite{Nikolov2018}, \mcite{Zimmer2018}] and the references therein.

Clearly, the squeezing function is a biholomorphic invariant. Moreover, it is known~\cite[Theorem~2.1]{DengGuanZhang2012} that the supremum in the definition of the squeezing function is always attained; in other words there exists an \textsl{extremal mapping}, i.e. an injective holomorphic map $f:\Omega\to\B$ such that ${f(z)=0}$ and $\dist\big(0,\partial f(\Omega)\big)=S_\Omega(z)$.

Recently, Ng, Tang and Tsai \cite{NgTangTsai} have determined the squeezing function for an annulus $\A_r:={\{z \in \C \colon  r<|z|<1\}}$, and they have formulated a conjectural formula for the squeezing function of planar domains of higher (but finite) connectivity.
The aim of this paper is the construction of a counterexample to the conjecture of Ng, Tang and Tsai for domains of connectivity beyond two.
Moreover, we give a simple proof of their result for the doubly connected case. Unlike the approach in \cite{NgTangTsai}, which is based on the Loewner differential equation on an annulus,
we use only rather elementary potential-theoretic reasoning. An advantage of our method is that it allows us to identify \textit{all} extremal functions.

\section{Main results}
In order to state our results we briefly recall some basic concepts.
A \textsl{circularly slit disk} is a subdomain $D$ of the unit disk $\D:=\{z\in\C\colon|z|<1\}$ containing the origin such that $\overline{\D} \setminus D$ consists of the unit circle $\partial \D$ and (closed)
arcs lying on concentric circles centered at the origin. It is admissible that the arcs degenerate to points.

\begin{remark}\label{RM_canonica-map-existence}
Let $\Omega$ be a domain in $\C$ with at least one non-degenerate (i.e., different from a singleton) boundary component. It is known, see e.g. \cite{Tsuji1975, Grunsky1978, ReichWarschawski1960} that for any such domain $\Omega$ and any fixed ${z \in \Omega}$ there is always a conformal mapping of~$\Omega$ onto some circularly slit disk that takes $z$ to~$0$. If the domain $\Omega$ is \textit{finitely connected} with non-degenerate boundary components $\Gamma_0,\Gamma_1,\ldots,\Gamma_n$, say, then
for each ${z \in \Omega}$ and each ${j=0,1,\ldots,n}$ we can find a unique conformal map $f_{z,j}$ of $\Omega$ onto a circularly slit disk  normalized by ${f_{z,j}(z)=0}$,  ${f'_{z,j}(z)>0}$, and $f_{z,j}(\Gamma_j)=\partial\UD$. Note that, in general, $f_{z,j}$ does not have to admit a continuous extension to~$\partial\Omega$; writing $f(\Gamma_j)$, where $f$ is any conformal map of~$\Omega$, we mean the boundary component $C_j$ of $f(\Omega)$ that corresponds to~$\Gamma_j$ under the map~$f$ in the following sense: if $(z_k)$ is a sequence of points in $\Omega$ which converges to a point on~$\Gamma_j$, then every limit point of the sequence $(f(z_k))$ belongs to~$C_j$.
\end{remark}

In the notation and terminology of  Remark~\ref{RM_canonica-map-existence} the following conjecture was formulated
 in~\cite{NgTangTsai}.
\begin{conjecture}\label{CNJ}
Let $\Omega\subset\C$ be an $m$-connected domain without degenerate boundary components. Then for any~${z\in\Omega}$ the squeezing function $S_\Omega$ is equal to
\begin{equation}\label{EQ_conjectured}
S_\Omega(z)~=\max_{j=0,\ldots,m-1}\dist\big(0,\partial f_{z,j}(\Omega)\big).
\end{equation}
\end{conjecture}

\noindent \hbox to\textwidth{Our main result states that this conjecture fails if the connectivity of~$\Omega$ is higher than two.}
\begin{theorem}\label{TH_tripply}
For each $m\ge3$, there exists an $m$-connected   domain~$\Omega\subset\C$ without degenerate boundary components  and a point ${z\in\Omega}$ such that~\eqref{EQ_conjectured} does not hold.
\end{theorem}
Conjecture~\ref{CNJ} is however true in the doubly connected case. This is the main result of~\cite{NgTangTsai}.
The following theorem gives slightly more precise information by identifying \textit{all} extremal functions.
\begin{theorem}\label{TH_doubly}
Formula~\eqref{EQ_conjectured} holds for any doubly connected domain~$\Omega$ with at least one non-degenerate boundary component and for any ${z\in\Omega}$.
Each extremal function is a conformal map onto a circularly slit disk.
In particular, for any $r\in(0,1)$ and any~$z\in\A_r$,
\begin{equation}\label{EQ_Squeezing-for-the-annulus}
S_{\A_r}(z)=\max\{|z|\,,r/|z|\}.
\end{equation}
\end{theorem}

The proof of (\ref{EQ_Squeezing-for-the-annulus}) in \cite{NgTangTsai} is based on a representation of conformal maps  in terms of the Loewner differential equation. It is well--known that one disadvantage of the Loewner method is its possible failure to identify all of the extremal functions. Our  proof of Theorem \ref{TH_doubly} is based on potential theory and more or less automatically gives complete description of the extremal functions.

\begin{remark}\label{RM_extremal-property}
  It is worth  mentioning that there is an infinitesimal version of Conjecture~\ref{CNJ} which in fact \textit{does} hold  for any finitely connected domain and which, incidentally, will be  one of the key ingredients for the proof of Theorem~\ref{TH_tripply}. To state this result, we fix $z\in\Omega$ and ${j\in\{0,\ldots,m-1\}}$ and consider the set of  all injective holomorphic functions $f:\Omega\to\UD$ normalized by ${f(z)=0}$, ${f'(z)>0}$ and such that $f(\Gamma_j)$ is the outer boundary of~$f(\Omega)$, i.e. $f(\Gamma_j)$ is the boundary of the unbounded component of $\C \setminus f(\Omega)$. Then the  maximum of $f'(z)$ is achieved for $f=f_{z,j}$ and only for this function.
  To prove this remarkable fact, we notice that if $f$ maximizes $f'(z)$, then according to the Schwarz Lemma, the outer boundary of $f(\Omega)$ must be the unit circle $\partial\UD$. It remains to apply the following well-known result for ${D:=f(\Omega)}$ and ${\varphi:=f_{z,j}\circ f^{-1}}$.
\end{remark}

\begin{proposition}[\protect{Tsuji \cite[Lemma~2(i) on p.\,409]{Tsuji1975}}]\label{PR_Tsuji}
Let $D\subset\UD$ be a finitely connected domain with outer boundary~$\partial\UD$ and let $\varphi$ be the conformal mapping of~$D$ onto a circularly slit disk normalized by $\varphi(0)=0$, $\varphi'(0)>0$, and $\varphi(\partial\UD)=\partial\UD$. Then $\varphi'(0)\ge1$, with $\varphi'(0)=1$ if and only if $D$ is a circularly slit disk, in which case $\varphi=\id_D$.
\end{proposition}

The paper is organized as follows. In Section~\ref{S_auxil} we describe the potential--theoretic tools on which our work is based, namely  harmonic measure, logarithmic potentials and conformal mappings as well as their behaviour  w.r.t.~kernel convergence. This section contains several auxiliary statements which are either new or otherwise only implicitly contained in the vast literature on the subject.  Strictly speaking, for the purpose of this paper, some of these results would only be needed for domains of connectivity two or for domains with degenerate boundary components.
However, for the sake of clarity and consistency, and in view of  potential further applications, we state and prove these auxiliary results in their natural setting for domains of any  finite connectivity. In Section~\ref{S_doubly} we prove \hbox{Theorem~\ref{TH_doubly}}. The proof is based on the doubly connected case of Theorem~\ref{TH_harmonic-measure-via-Log-potential} in Section~\ref{S_auxil}, which expresses the harmonic basis (i.e. the harmonic measures of boundary components) for a finitely connected domain~$\Omega$ in terms of logarithmic potentials of specific positive Borel measures which are supported on the individual boundary components of $\Omega$. In \hbox{Section~\ref{S_fx}} we discuss in detail the mapping properties and dependence on parameters of the canonical conformal mapping of the standard annulus $\A_r$ onto a circularly slit disk. Our treatment is based on the Schottky\,--\,Klein prime function. In the final Section~\ref{S_tripply}, these mapping properties are then  combined with the results of Section~\ref{S_auxil} to prove Theorem~\ref{TH_tripply}.

Throughout the paper we will denote by $\D(a,\rho)$ the open disk ${\{z\in\C\colon |z-a|<\rho\}}$.

\section{Harmonic measure, logarithmic potentials, and kernel convergence} \label{S_auxil}
In this section we suppose that $\Omega\subset\C$ is a \textit{bounded} finitely connected domain  with $n+1$ non-degenerate boundary components $\Gamma_j$, $j=0,1,\ldots,n$.

For $z\in\Omega$, we denote by $\omega_\Omega(z,\cdot)$ the \textsl{harmonic measure} for a point~$z\in\Omega$ relative to the domain~$\Omega$. For the definition and fundamental properties of the harmonic measure we refer the reader to \cite[\S4.3]{Ransford}.

\begin{remark}\label{RM_Perron-solution-is-continuous}
Finite sets are removable for bounded harmonic functions, see e.g. \cite[Corollary~1.5 on p.\,73]{ConwayII}. This well-known fact together with Perron's theory of the Dirichlet problem for harmonic functions \cite[Theorems~4.1.5 and~4.2.2]{Ransford} imply that under the above assumptions, for any continuous function ${\varphi:\partial\Omega\to\Real}$ there exists a unique continuous function $\psi:\overline\Omega\to\Real$ which is harmonic in~$\Omega$ and which coincides with~$\varphi$ on every non-degenerate boundary component of~$\Omega$. This function is given by the \textsl{generalized Poisson integral}
\begin{equation}\label{EQ_generalized-Poisson}
\psi(z)=\int\limits_{\partial\Omega}\varphi(\zeta)\,\omega_{\Omega}(z,\mathrm{d}\zeta)= \int\limits_{P}\varphi(\zeta)\,\omega_{\Omega}(z,\mathrm{d}\zeta)\quad\text{for all~$z\in\Omega$},
\end{equation}
where $P:=\Gamma_0\cup\Gamma_1\cup\ldots\cup\Gamma_n$ is the union of all non-degenerate boundary components of~$\Omega$.
The Maximum Principle asserts in this case, see e.g. \cite[Theorem~4.1.2]{Ransford}, that $$\max\{\psi(z)\colon z\in\overline\Omega\}=\max\{\varphi(\zeta)\colon \zeta\in P\}.$$
\end{remark}

\begin{remark}\label{RM_Greens-function}
Formula~\eqref{EQ_generalized-Poisson} leads to the following representation for the Green function $G_\Omega$ of the domain~$\Omega$,
\begin{equation}\label{EQ_Greens-function}
G_\Omega(z,w)=\log\frac1{|z-w|}\,-\int_P\log\frac1{|z-w|}\,\,\omega_{\Omega}(z,\mathrm{d}\zeta)\quad\text{for all~$z,w\in\Omega$, $z\neq w$}.
\end{equation}
For any fixed $w\in\Omega$, we will assume that $G_\Omega(\cdot,w)$ is extended to~$\partial\Omega$ by continuity.
\end{remark}

To state the first theorem of this section we need to introduce some more notation.
Denote by $K_j$ the connected component of ${\C\setminus \Omega}$ bounded by $\Gamma_j$. We adopt the convention that, unless explicitly stated otherwise, the boundary components are labelled such  that $\Gamma_0$ is the \textsl{outer boundary} of~$\Omega$, i.e. the component $K_0$ is the unbounded one.

Let $\omega_j(z):=\omega_\Omega(z,\Gamma_j)$, $j=0,\ldots,n$. Note that $\omega_j$ is the unique harmonic function in~$\Omega$ admitting a continuous extension to~$\overline\Omega$ with ${\omega_j|\Gamma_j\equiv 1}$ and ${\omega_j|\Gamma_k\equiv 0}$ for $k\neq j$.

Consider the integrals
$$
 \lambda_{jk}:=\frac{1}{2\pi}\int\limits_{\partial\mathcal D_k}\frac{\partial \omega_j}{\partial\mathrm{n}}\,\di s,\quad k=0,\ldots, n,
$$
where $\mathcal D_k$ is a Jordan domain (unbounded in case ${k=0}$) with $C^1$-smooth boundary ${\partial \mathcal D_k\subset\Omega}$ and such that $\mathcal D_k\cap\big(\Complex\setminus\Omega\big)=K_k$. Here $\partial/\partial\mathrm{n}$ stands for the derivative along the \textit{inner} normal w.r.t.~$\mathcal D_k$.
Note that
\begin{equation}\label{EQ_sum-of-periods}
 \sum_{j=0}^n\lambda_{jk}=0,\quad k=0,\ldots,n,
\end{equation}
because $\sum_{j=0}^n\omega_{j}=\omega_\Omega(\cdot,\partial\Omega)\equiv1$. Moreover, thanks to Green's formula, \begin{equation}\label{EQ_sum-of-periods-bis}
 \lambda_{j0}=-\sum_{k=1}^n\lambda_{jk},\quad j=0,\ldots,n.
\end{equation}
The functions $\omega_1,\ldots,\omega_n$ form the so-called \textsl{harmonic basis} in the domain~$\Omega$ and the numbers $\lambda_{jk}$ are known as \textsl{periods}; for more details see, e.g., \cite[\S15.1]{ConwayII}. To fix the terminology in a more precise way, we will say that $\lambda_{jk}$ is the period of~$\omega_j$ \textsl{associated with the boundary component~$\Gamma_k$}.

\begin{remark}\label{RM_harm-meas-conformal inv}
The harmonic basis and the period matrix are conformally invariant in the following sense. Let $f$ be a conformal mapping of~$\Omega$ onto another bounded domain in~$\C$. Then for any $j\in\{0,\ldots,n\}$ and all ${z\in\Omega}$,
\begin{equation}\label{EQ_harm-meas-conformal inv}
\omega_{f(\Omega)}\big(f(z),f(\Gamma_j)\big)=\omega_{\Omega}(z,\Gamma_j)=\omega_j(z).
\end{equation}
It follows easily that for any $k\in\{0,\ldots,n\}$ the period of $\omega_{f(\Omega)}\big(\,\cdot\,,f(\Gamma_j)\big)$ associated with~$f(\Gamma_k)$ equals $\lambda_{jk}$.

To establish equality~\eqref{EQ_harm-meas-conformal inv}, it is sufficient to recall that $\omega_{f(\Omega)}\big(\,\cdot\,,f(\Gamma_j)\big)$ and $\omega_{\Omega}(\cdot,\Gamma_j)$ extend continuously to the boundaries of $\Omega$ and $f(\Omega)$, respectively,  and notice that for any $k\in\{0,\ldots,n\}$ and any sequence ${(z_n)\subset\Omega}$, $\dist(z_n,\Gamma_k)\to0$ as ${n\to+\infty}$ if and only if $\dist\big(f(z_n),f(\Gamma_k)\big)\to 0$ as ${n\to+\infty}$.
\end{remark}

\begin{remark}\label{RM_period-matrix}
It is known that the \textsl{period matrix} $\Lambda_0:=[\lambda_{jk}]_{1\le j,k \le n}$ is invertible and symmetric.  The proof of this fact for smooth boundaries can be found, e.g., in \cite[Proposition~1.7 on p.\,74]{ConwayII} and \cite[p.\,39]{Nehari}. Since any finitely connected domain can be mapped conformally onto a domain with smooth boundary, the general case holds thanks to Remark~\ref{RM_harm-meas-conformal inv}.
Moreover, equalities \eqref{EQ_sum-of-periods} and~\eqref{EQ_sum-of-periods-bis} show that the extended period matrix $\Lambda:=[\lambda_{jk}]_{0\le j,k \le n}$ is symmetric as well and that for any $m\in\{0,\ldots,n\}$,
the matrix $\Lambda_m:=[\lambda_{jk}]_{j,k\in J_m}$, where $J_m:=\{0,\ldots,n\}\setminus\{m\}$, is invertible.
\end{remark}

Below we will see that the harmonic functions~$\omega_1,\ldots,\omega_n$ can be represented in terms of logarithmic potentials.
For a finite Borel measure $\mu$ with compact support in~$\C$, the \textsl{logarithmic potential} $V_\mu$ is defined by
$$V_\mu(w):=\int\log\frac{1}{|w-z|}\,\di\mu(z).$$
Note that $V_\mu$ is a harmonic function in $\C\setminus\supp\mu$. Moreover,
\begin{equation}\label{EQ_potential-asymptotic}
V_\mu(w)=|\mu|\log\frac{1}{|w|}+O(1/|w|)\quad\text{as~$w\to\infty$},
\end{equation}
where $|\mu|$ denotes for the total mass of $\mu$, i.e. $|\mu|:=\mu(\supp\mu)$.
See, e.g., \cite[Sect.\,3.1]{Ransford} for more details.

The following theorem is certainly known to the specialists. Since we have not been able to trace any suitable precise reference, we  state it here and include  a proof.
\begin{theorem}\label{TH_harmonic-measure-via-Log-potential}
 In the above notation, for each $j=1,\ldots,n$, there exist finite positive Borel measures  $\mu_{j0},\ldots,\mu_{jn}$ supported on $\Gamma_0,\ldots,\Gamma_n$, respectively, such that
\begin{equation}\label{EQ_harmonic-measure-via-Log-potential}
\omega_j(z)~=~V_{\mu_{jj}}(z)~-\sum_{\substack{0\le k\le n\\[.3ex] k\neq j}} V_{\mu_{jk}}(z)\qquad\text{for all~$z\in\Omega$}.
\end{equation}
The measures $\mu_{jk}$ are uniquely defined by $\Omega$. Moreover,
\begin{align}
\label{EQ_supp}
 &\supp\mu_{jk}=\Gamma_k,\quad k=0,\ldots,n,\\[1ex]
 \label{EQ_charges}
 &|\mu_{jk}|=|\lambda_{jk}|, \quad k=0,\ldots,n,\\[1ex]
 \label{EQ_charges-signes}
 &\lambda_{jj}>0\quad\text{and}\quad\lambda_{jk}<0~\text{~for all~$~k=0,\ldots,n$, $k\neq j$}.
\end{align}
\end{theorem}
\begin{proof}
The functions~$\omega_0,\omega_1,\ldots,\omega_n$ extend harmonically to every isolated point of~$\partial\Omega$. Therefore, without loss of generality we may suppose that $\Omega$ has no degenerate boundary components, i.e. $\partial\Omega=\Gamma_0\cup\Gamma_1\cup\ldots\cup\Gamma_n$.

Fix $j\in\{1,\ldots,n\}$ and consider the function $u:\Complex\to\Real$ defined by
$$
u(z)=
 \begin{cases}
  \omega_j(z),&\text{if~$z\in\Omega$},\\
  0,          &\text{if~$z\in K_k$ with $k\in\{0,\ldots,n\}$, $k\neq j$},\\
  1,          &\text{if~$z\in K_j$}.
 \end{cases}
$$
This function is continuous in the whole plane, see e.g. \cite[Theorems~4.2.2 and~4.3.4]{Ransford}.
Clearly, $u$ is harmonic in $\C\setminus\partial \Omega$. Moreover, comparing $u(z_0)$ for points~$z_0\in\partial \Omega$ with the mean values of~$u$ over sufficiently small circles centered at~$z_0$ and taking into account that  $u(\C)\subset[0,1]$, one can easily see that $u$ is subharmonic in $\C\setminus K_j$  and superharmonic in $\C\setminus K$, where
$$
 K:=\bigcup_{\substack{0\le k\le n\\[.3ex] k\neq j}} K_k.
$$
Therefore, combining Riesz's Representation Theorem for subharmonic functions (see, e.g., \cite[Theorem~3.9 on p.\,104]{HaymanV1})  with the fact that the logarithmic potential of a finite Borel measure is harmonic in an open set~$A$ if and only if this measure vanishes on~$A$, one may conclude that there exist two uniquely defined finite positive Borel measures $\mu$ and $\nu$, supported on $\partial K_j=\Gamma_j$ and $\partial K=\bigcup_{k\neq j}\Gamma_k$, respectively, such that
$$
u=V_\mu-V_\nu+u_0,
$$
where $u_0$ is a harmonic function in~$\C$.  Put $\mu_{jj}:=\mu$ and $\mu_{jk}=\nu|_{\Gamma_k}$ for ${k=0,\ldots,n}$, $k\neq j$.

By Gau{\ss}' Theorem (see, e.g., \cite[Theorem~1.1 on p.\,83]{SaffTotik}) we have
$$
 \lambda_{jk}=
 \begin{cases}
  \hphantom{-}|\mu_{jj}|,&\text{if $k=j$},\\[.45ex]
  -|\mu_{jk}|,&\text{if~$k\in\{1,\ldots,n\}$, $k\neq j$.}
 \end{cases}
$$
This proves~\eqref{EQ_charges} and the non-strict inequalities in~\eqref{EQ_charges-signes} for $k\neq0$. The fact that $\lambda_{jk}\neq0$ for all ${k\in\{0,\ldots,n\}}$ follows from~\eqref{EQ_supp}, which we will prove below.

Applying Gau{\ss}' Theorem in the disk $\UD(0,R)$ for $R>0$ large enough, we see that $|\mu|=|\nu|$. With~\eqref{EQ_sum-of-periods} taken into account, it follows that \eqref{EQ_charges} and~\eqref{EQ_charges-signes}, again with the non-strict inequality, hold also for~${k=0}$.

Moreover, thanks to~\eqref{EQ_potential-asymptotic}, we have ${V_{\mu}(z)-V_{\nu}(z)\to0}$ as $z\to\infty$. By  construction, $u$ vanishes identically in a neighbourhood of~$\infty$. Hence, applying the Maximum Principle to the harmonic function $u_0=u-(V_{\mu}-V_{\nu})$, we see that $u_0\equiv0$. This proves~\eqref{EQ_harmonic-measure-via-Log-potential}.

To show that $\supp \mu=\Gamma_j$ suppose on the contrary that there exists a neigbourhood~$U$ of a point~$z_0\in \Gamma_j$ such that $\mu(U)=0$. Replacing $U$, if necessary, with a smaller neighbourhood, we may suppose that ${U\cap K=\emptyset}$.  Then $u$ is harmonic in~$U$; moreover, $u(z)\le 1$ for all ${z\in U}$ and $u(z_0)=1$. By the Maximum Principle $u(z)=1$ for all $z\in U$. Since $U\cap \Omega\neq\emptyset$, this contradicts the fact that $u$ is not constant in~$\Omega$. The equality $\supp \nu=\bigcup_{k\neq j}\Gamma_k$ can be proved in a similar way.
\end{proof}

\begin{remark}\label{RM_several components}
For a non-empty set $J\subset\{1,\ldots,n\}$ denote
$$\Gamma_J:=\bigcup_{j\in J}\Gamma_j \quad\text{and}\quad J^*:=\{0,\ldots,n\}\setminus J.$$
Repeating the argument of the above proof with ${u:\Complex\to[0,1]}$ defined by ${u(z):=\omega_\Omega(z,\Gamma_J)}$ if ${z\in\Omega}$, ${u(z):=1}$ if ${z\in K_j}$ for some ${j\in J}$, and ${u(z):=0}$ if ${z\in \Gamma_k}$ with ${k\in J^*}$, we can conclude that
\begin{gather}\label{EQ_several components}
\omega_\Omega(\,\cdot\,,\Gamma_J)=\sum_{j\in J}\omega_j=V_\mu-V_\nu,\\
\intertext{where}\notag
\mu:=\sum_{j\in J}\Big(\mu_{jj}\,-\!\sum_{k\in J\setminus\{j\}}\mu_{jk}\Big)=\sum_{k\in J}\Big(\mu_{kk}\,-\!\sum_{j\in J\setminus\{k\}}\mu_{jk}\Big)\quad\text{and}\quad \nu:=\sum_{j\in J}\,\sum_{k\in J^*}\mu_{jk}
\end{gather}
are \textit{positive} finite Borel measures supported on $\Gamma_J$ and $\bigcup_{k\in J^*}\Gamma_k$, respectively, with
\begin{equation}\label{EQ_several components:charges}
|\mu|\,=\,|\nu|\,=\sum_{j,k\in J}\lambda_{jk}\,=\,-\sum_{j\in J}\sum_{k\in J^*}\lambda_{jk}\,=\sum_{j,k\in J^*}\lambda_{jk}\,.
\end{equation}
(Here we have also taken into account relations~\eqref{EQ_sum-of-periods} and~\eqref{EQ_sum-of-periods-bis}.)
\end{remark}
\begin{remark}\label{RM_repres-extends-to-bndry}
It is clear from the  proof of Theorem~\ref{TH_harmonic-measure-via-Log-potential} that the representation~\eqref{EQ_harmonic-measure-via-Log-potential}, as well as formula~\eqref{EQ_several components}, are valid also on the boundary of~$\Omega$.
\end{remark}

As a simple application of Theorem~\ref{TH_harmonic-measure-via-Log-potential}, we recover the following known result, see e.g.~\cite[Proof of Lemma~6.3, pp.\,97-98]{ConwayII}.
\begin{corollary}\label{CR_radii}
Let $m\in\{0,1,\ldots,n\}$ and $z_0\in\Omega$, and let $f$ be a conformal mapping of~$\Omega$ onto a suitable circularly slit disk
$$
 D:=\UD\,\big\backslash\Big(\bigcup\limits_{\substack{0\le k\le n\\[.3ex] k\neq m}} C_k\Big)
$$
with $f(z_0)=0$, $f(\Gamma_m)=\UC$, and $f(\Gamma_k)=C_k$ for all ${k\neq m}$. Then the radii $r_k$ of the circular arcs $C_k$ satisfy the following system of linear equations:
\begin{equation}\label{EQ_radii}
  \sum\limits_{\substack{0\le k\le n\\[.3ex] k\neq m}}\lambda_{jk}\log \frac1{r_k}=\omega_j(z_0),
  \quad 0\le j\le n,~ j\neq m.
\end{equation}
\end{corollary}
\begin{remark}
Note that the linear system~\eqref{EQ_radii} has a unique solution because its coefficient matrix is invertible; see Remark~\ref{RM_period-matrix}.
\end{remark}
\begin{proof}[\protect{Proof of Corollary~\ref{CR_radii}}]
Since the harmonic basis and the periods $\lambda_{jk}$, ${j,k\in\{0,\ldots,n\}}$, are conformally invariant, see Remark~\ref{RM_harm-meas-conformal inv}, we may suppose that $\Omega=D$, $m=0$, and $f=\id_\Omega$. Accordingly, we have $z_0=0$.
Fix ${j\in\{1,\ldots,n\}}$ and apply Theorem~\ref{TH_harmonic-measure-via-Log-potential}. By~\eqref{EQ_harmonic-measure-via-Log-potential},
\begin{equation}\label{EQ_radii-proof1}
\omega_j(0)=|\mu_{jj}|\log\frac1{r_j}~-\sum_{\substack{1\le k\le n\\[.3ex] k\neq j}} |\mu_{jk}|\log\frac1{r_k},
\end{equation}
where we took into account that $\mu_{jk}$, $k=1,\ldots,n$, are supported on circles of radius~$r_k$ centred at the origin and that $\mu_{j0}$ is supported on~$\UC$ and hence $V_{\mu_{j0}}(0)=0$.

Taking into account~\eqref{EQ_charges} and~\eqref{EQ_charges-signes}, from~\eqref{EQ_radii-proof1} we obtain
$$
\omega_j(0)=\sum_{k=1}^n\lambda_{jk}\log\frac{1}{r_k},
$$
as desired.
\end{proof}

In the doubly connected case, which is the relevant case for the proof of Theorem~\ref{TH_doubly}, the above results allow us to deduce the following statement.

\begin{corollary}\label{CR_doubly-conn}
Let $r\in(0,1)$ and let $f$ be a conformal mapping of $\A_r$ onto a bounded domain~$\Omega$ such that $\partial\UD$ corresponds under~$f$ to the outer boundary~$\Gamma_0$ of~$\Omega$.  Then the following two assertions hold.
\begin{itemize}
\item[(A)]
 For any $\zeta\in\Omega$,
\begin{equation}\label{EQ_doubly-conn}
\log\frac{1}{|f^{-1}(\zeta)|}=V_{\mu_*}(\zeta)-V_{\nu_*}(\zeta),
\end{equation}
where $\mu_*$ and $\nu_*$ are suitable \textsl{probability} measures with $$\supp\nu_*=\Gamma_0\quad\text{and}\quad \supp\mu_*=\Gamma_1:=\partial\Omega\setminus\Gamma_0.$$
\item[(B)] Moreover, if $\Omega$ is a circularly slit disk, then $\Gamma_1\subset\big\{\zeta\colon|\zeta|=|f^{-1}(0)|\big\}$.
\end{itemize}
\end{corollary}
\begin{proof}
Apply Theorem~\ref{TH_harmonic-measure-via-Log-potential} with $n=j=1$. By Remark~\ref{RM_harm-meas-conformal inv}, ${\omega_1(\zeta)=\log|f^{-1}(\zeta)|/\log r}$ for all ${\zeta\in\Omega}$.
Thanks to the same remark, in order to find $\lambda_{1,1}$, we may suppose that $f=\id_{\A_r}$. In this way we see that the measures $\mu_{1,1}$ and $\mu_{1,0}$ in formula~\eqref{EQ_harmonic-measure-via-Log-potential} satisfy
$$
|\mu_{1,1}|=|\mu_{1,0}|=\lambda_{1,1}=\frac1{\log(1/r)}.
$$
Hence to prove (A), it remains to set ${\mu_*:=(1/\lambda_{1,1})\mu_{1,1}}$ and ${\nu_*:=(1/\lambda_{1,1})\mu_{1,0}}$.

Assertion (B) is a well known fact, see e.g. \cite[Lemma~2.2]{NgTangTsai} or \cite[Lemma~3]{ReichWarschawski1960}. Alternatively, one can use system~\eqref{EQ_radii} in Corollary~\ref{CR_radii}, which reduces in our case to the unique equation $\lambda_{1,1}\log(1/r_1)=\omega_1\big(f^{-1}(0)\big)$ and hence yields ${r_1=|f^{-1}(0)|}$.
\end{proof}

\begin{remark}
It is possible to show that the measure $\mu_*$ in Corollary~\ref{CR_doubly-conn} coincides with the so-called \textsl{Green equilibrium distribution} on $K_1$ relative to the simply connected domain $D:=\Omega\cup K_1=\C\setminus K_0$, where in accordance with the notation introduced at the beginning of this section, $K_1$ and $K_0$ stand for the bounded and unbounded connected components of $\C\setminus\Omega$, respectively. Formula~\eqref{EQ_doubly-conn} can be rewritten as
$$
\log\frac{1}{|f^{-1}(\zeta)|}=\int\limits_{\Gamma_1}G_D(\zeta,w)\di\mu_*(w),\quad \zeta\in D\setminus K_1,
$$
where $G_D$ is the Green function of the domain~$D$. For more details on Green equilibrium distributions, see e.g. \cite[Sect.\,II.5]{SaffTotik}, or \cite[p.\,94-95]{Tsuji1975} where the case~$D=\UD$ is considered.
\end{remark}

The classical Kernel Convergence Theorem due to Carath\'eodory, see e.g. \cite[Theorem~1.8 on p.\,14]{Pommerenke:BoundaryBehaviour},  relates the limit behaviour of a sequence of hyperbolic \textit{simply connected} domains with the convergence of the corresponding conformal mappings onto the canonical domain (the unit disk).
The notion of kernel convergence extends naturally to multiply connected domains, but no complete analogue of Carath\'eodory's result seems to be known  even for the finitely connected case.
Considerable progress in this direction has been made in~\cite{Binder, Comerford}. In this regard, it is worth mentioning that the canonical mappings of a multiply connected domain are closely related to its harmonic basis, see e.g. \cite[Chapter~15]{ConwayII}.
However, to the best of our knowledge, no known results apply to the special, but rather interesting, case in which the boundary consists of a non-empty constant part plus a variable part, which shrinks in the limit to a single point. The following theorem describes the limit behaviour of the harmonic measure, Green's function, and the conformal mappings onto circularly slit disks in this special case.

\begin{theorem}\label{TH_shrinking-conv}
Let $\Omega_0$ and $\Omega_n\varsubsetneq\Omega_0$, $n\in\Natural$, be bounded finitely connected domains in~$\C$ and let $\zeta^*\in \Omega_0$. Suppose that $$\sup\big\{|\zeta-\zeta^*|\colon \zeta\in\Omega_0\setminus\Omega_n\big\}\to0\quad \text{as~$~n\to+\infty$.}$$ Then the following five statements hold.
\begin{itemize}
\item[(A)] If $\varphi$ is continuous on $\partial\Omega_0$ and in a neighbourhood of~$\zeta^*$, then for all ${z\in\Omega_0\setminus\{\zeta^*\}}$,
\begin{equation}\label{EQ_shrinking-harm-meas}
\int\limits_{\partial\Omega_n}\varphi(\zeta)\,\omega_{\Omega_n}(z,\mathrm{d}\zeta)~\to~
\int\limits_{\partial\Omega_0}\varphi(\zeta)\,\omega_{\Omega_0}(z,\mathrm{d}\zeta)\quad\nlim.
\end{equation}
Moreover, for any~$\varepsilon>0$, the convergence in~\eqref{EQ_shrinking-harm-meas} is uniform in $\Omega_0\setminus\UD(\zeta^*,\varepsilon)$.
\smallskip
\item[(B)] For any~$\varepsilon>0$, the sequence $(G_{\Omega_n})$ converges  to~$G_{\Omega_0}$ uniformly in
$$
  \big\{(z,w)\colon z\in\!\overline{\,\Omega_0\!}\,\setminus \UD(\zeta^*,\varepsilon),~w\in\Omega_0\setminus \UD(\zeta^*,\varepsilon),~z\neq w\big\}.
$$
\item[(C)] Let $\Gamma$ be a non-degenerate boundary component of $\Omega_0$ and let $(z_n\in\Omega_n)$ be a sequence converging to some $z_0\in\Omega_0\setminus\{\zeta^*\}$. For each ${n\in\{0\}\cup\Natural}$, denote by $f_n$ the unique conformal mapping of $\Omega_n$ onto a circularly slit disk normalized by $f_n(z_n)=0$, $f_n'(z_n)>0$, $f_n(\Gamma)=\partial\UD$. Then the sequence~$(f_n)$ converges locally uniformly in~$\Omega_0\setminus\{\zeta^*\}$ to~$f_0$.
\smallskip
\item[(D)] For each $n\in\Natural$,  let $g_n$ be a conformal mapping of $\Omega_n$ with ${g_n(\Omega_n)\subset\UD}$. Further, suppose that for any $n\in\Natural$ the outer boundary of $g_n(\Omega_n)$ corresponds under $g_n^{-1}$ to a subset of~$\partial\Omega_n\setminus\partial\Omega_0$. If the sequence~$(\zeta_n):=\big(g_n^{-1}(0)\big)$ is contained in a compact subset of $\Omega_0\setminus\{\zeta^*\}$, then the sequence $(g_n)$ converges locally uniformly in~$\Omega_0\setminus\{\zeta^*\}$ to~$g_0\equiv0$.
\smallskip
\item[(E)] For the sequences $(f_n)$ and $(g_n)$ defined above, we have
\begin{align}
\label{EQ_shrinking-dist1}
&\dist\big(0,\partial f_n(\Omega_n)\big)\to\dist\big(0,\partial f_0(\Omega_0\setminus\{\zeta^*\})\big)
\quad\text{and} \\
\label{EQ_shrinking-dist2}
&\dist\big(0,\partial g_n(\Omega_n)\big)\to0 \quad\qquad\qquad\qquad\text{as $~n\to+\infty$.}
\end{align}
\end{itemize}
\end{theorem}
\begin{remark}
  Using Moebius transformations it is easily seen that the above theorem holds also for unbounded domains $\Omega_0\subset\C$ with external points. A similar approach allows one to extend Theorem~\ref{TH_shrinking-conv} to the case of an unbounded domain without external points, provided it has at least one non-degenerate boundary component. In such a case, however, the argument becomes slightly more complicated because one should use a conformal map of the form $z\mapsto\sqrt{(z-a)/(z-b)}$, which does not extend
  to a one-to-one map on the boundary.
\end{remark}
\begin{proof}[\protect{Proof of Theorem~\ref{TH_shrinking-conv}}]
Denote $\Upsilon_n:=\partial\Omega_n\setminus\partial\Omega_0$. We are going to show that:\\[1ex]
\textsc{Claim 1.} For any ${\varepsilon>0}$, as ${n\to+\infty}$, $\omega_{\Omega_n}(z,\Upsilon_n)\to 0$  uniformly in $\moverline{\Omega_0}\setminus\UD(\zeta^*,\varepsilon)$.
\medskip

Denote by $P_0$ the union of all non-degenerate boundary components of~$\Omega_0$ and let~$P_n$, ${n\in\Natural}$, stand for the union of all non-degenerate connected components of~$\Upsilon_n$.
If ${P_n=\emptyset}$, then $\omega_{\Omega_n}(z,\Upsilon_n)\equiv0$. Therefore, we may suppose that ${P_n\neq\emptyset}$ for any ${n\in\Natural}$. Then Theorem~\ref{TH_harmonic-measure-via-Log-potential} together with Remark~\ref{RM_several components} implies that for each ${n\in\Natural}$,
\begin{equation}\label{EQ_harm-measure-of-Upsilon_n}
\omega_{\Omega_n}(z,\Upsilon_n)=V_{\mu_n}(z)-V_{\nu_n}(z)\quad\text{for all~$z\in\moverline{\Omega_n}$},
\end{equation}
where $\mu_n$ and $\nu_n$ are suitable positive Borel measures supported on~$P_n$ and $P_0$, respectively. Moreover,
   $${c_n:=|\mu_n|=|\nu_n|.}$$
It follows from the very definition of the logarithmic potential that
\begin{equation*}
\inf_{z\in P_n}V_{\mu_n}(z)\ge -c_n\log\diam(P_n)\quad\text{and}\quad \sup_{z\in P_n}V_{\nu_n}(z)\le -c_n\log\dist(P_0,P_n)
\end{equation*}
for all $n\in\Natural$. Taking into account that $V_{\mu_n}(z)-V_{\nu_n}(z)=1$ for all ${z\in P_n}$, we deduce that
\begin{equation}\label{EQ_c_n-to-zero}
1/c_n\ge\log\frac{\dist(P_0,P_n)}{\diam(P_n)}\to+\infty,\quad\text{i.e. $c_n\to0$,$~$ as $~n\to+\infty$.}
\end{equation}
Using~\eqref{EQ_harm-measure-of-Upsilon_n} we see that for any $\varepsilon>0$  with $\partial\UD(\zeta^*,\varepsilon)\subset\Omega_0$,  ${\omega_{\Omega_n}(\Upsilon_n,z)\to0}$ uniformly on $\partial\UD(\zeta^*,\varepsilon)$ as ${n\to+\infty}$. Now Claim~1 follows easily from the Maximum Principle, see Remark~\ref{RM_Perron-solution-is-continuous}, applied to the harmonic functions ${\omega_{\Omega_n}(\cdot,\Upsilon_n)}$ in the domain $\Omega_0\setminus\!\overline{\,\UD(\zeta^*,\varepsilon)}$.

\medskip
\noindent\textsc{Proof of (A)}. By Remark~\ref{RM_Perron-solution-is-continuous}, for each $n\in\Natural$ the function
$$
\psi_n(z)\,:=\,\int\limits_{\partial\Omega_n}\varphi(\zeta)\,\omega_{\Omega_n}(z,\mathrm{d}\zeta)~-~
\int\limits_{\partial\Omega_0}\varphi(\zeta)\,\omega_{\Omega_0}(z,\mathrm{d}\zeta)
$$
is harmonic in~$\Omega_n$ and extends continuously to the boundary of~$\Omega_n$ with $\psi_n|_{P_0}\equiv0$.
Hence, again by Remark~\ref{RM_Perron-solution-is-continuous}, for all $z\in\Omega_n$,
\begin{align*}
|\psi_n(z)|\,=\,\left|\int_{P_n}\psi_n(\zeta)\,\omega_{\Omega_n}(z,\mathrm{d}\zeta)\right|\,&\le \,\omega_{\Omega_n}(z,\Upsilon_n)\max\{|\psi(\zeta)|\colon\zeta\in P_n\}\\&\le
\,2\omega_{\Omega_n}(z,\Upsilon_n)\max\{|\varphi(\zeta)|\colon\zeta\in \partial\Omega_n\},
\end{align*}
where we have applied the triangle inequality and the Maximum Principle in order to estimate $|\psi_n|$ on~$P_n$.
Combined with Claim~1 the above inequality easily implies~(A).

\medskip
\noindent\textsc{Proof of (B)}. Let $\varepsilon>0$. Fix some $w\in\Omega_0\setminus\UD(\zeta^*,\varepsilon)$. Bearing in mind formula~\eqref{EQ_Greens-function}, consider the functions
\begin{align*}
\psi_n(z)~&:=~G_{\Omega_0}(z,w)\,-\,G_{\Omega_n}(z,w)\,=\\ &=\int\limits_{\partial\Omega_n}\log\frac{1}{|\zeta-w|}\,\,\omega_{\Omega_n}(z,\mathrm{d}\zeta)~-~
\int\limits_{\partial\Omega_0}\log\frac{1}{|\zeta-w|}\,\,\omega_{\Omega_0}(z,\mathrm{d}\zeta).
\end{align*}
For $n\in\Natural$ large enough, we have $P_n\subset\UD(\zeta^*,\varepsilon/2)$ and hence ${0\le G_{\Omega_0}(\zeta,w)\le M}$ for all ${\zeta\in P_n}$ and some constant $M>0$ depending on~$\varepsilon$ but not on~$w$. Since ${G_{\Omega_n}(\zeta,w)=0}$ for all ${\zeta\in P_n}$, we have $|\psi_n|\le M$ on~$P_n$, and it only remains to apply the argument used in the above proof of assertion~(A).

\medskip
\noindent\textsc{Proof of (C)}. Since $f_n(\Omega_n)\subset\UD$ for all ${n\in\Natural}$, the functions $f_n$ form a normal family in ${\Omega_0\setminus\{\zeta^*\}}$. Therefore, passing if necessary to a subsequence, we may suppose that $(f_n)$ converges locally uniformly in ${\Omega_0\setminus\{\zeta^*\}}$ to a holomorphic function ${f:\Omega_0\setminus\{\zeta^*\}\to\moverline\UD}$. Since $f$ is bounded in ${\Omega_0\setminus\{\zeta^*\}}$, it extends holomorphically to~$\zeta^*$.

To prove~(C) it is sufficient to show that~${f=f_0}$.
Clearly $f(z_0)=\lim_{n\to+\infty}{f_n(z_n)}=0$.
Moreover, $f'(z_0)\ge f_0'(z_0)$. Indeed, for each ${n\in\Natural}$, let $T_n$ be the automorphism of~$\UD$ such that ${T_n(f_0(z_n))=0}$ and ${T_n'(f_0(z_n))>0}$. Note that $T_n\circ f_0$ maps $\Omega_n$ conformally onto a subdomain of~$\UD$ in such a way that the boundary component~$\Gamma$ corresponds to~$\partial\UD$. Bearing in mind that ${(T_n\circ f_0)(z_n)=0}$ and ${(T_n\circ f_0)'(z_n)>0}$, by the extremal property of the conformal mappings onto circularly slit disks, see Remark~\ref{RM_extremal-property}, we have ${f_n'(z_n)\ge(T_n\circ f_0)'(z_n)}$.
Passing to the limit as ${n\to+\infty}$ and taking into account that $T_n\to\id_\UD$ yields the desired conclusion.
In particular,  this means that $f'(z_0)\neq0$. Hence $f$ is a conformal mapping of~$\Omega_0$ onto a subdomain of~$\UD$.

In order to show that $\Gamma$ corresponds under~$f$ to the outer boundary of~$f(\Omega_0)$, consider a $C^1$-smooth Jordan curve ${C\subset f(\Omega_0\setminus\{\zeta^*\})}$ that separates the outer boundary of $f(\Omega_0)$ from all other boundary components of $f\big(\Omega_0\setminus\{\zeta^*\}\big)$. Let $\mathcal D$ be the connected component of ${\C\setminus f^{-1}(C)}$ that contains~$\Gamma$. Then $\overline{\mathcal D}\cap\Omega_0\subset \Omega_n$ for all ${n\in\Natural}$ large enough and moreover, for all such $n$'s, $f_n(\mathcal D\cap\Omega_0)$ lies in the unbounded component of ${\Complex\setminus C_n}$, where $C_n:={f_n\big(f^{-1}(C)\big)}$, because ${f_n(\Gamma)=\partial\UD}$. Using this fact, we conclude that for any ${\zeta\in\mathcal D\cap\Omega_0}$,
\begin{align*}
\int\limits_{C}\frac{\di w}{w-f(\zeta)}~&=\int\limits_{f^{-1}(C)}\frac{f'(z)\di z}{f(z)-f(\zeta)}\\
&=\lim_{n\to+\infty}\int\limits_{f^{-1}(C)}\frac{f_n'(z)\di z}{f_n(z)-f_n(\zeta)}
=\lim_{n\to+\infty}\int\limits_{C_n}\frac{\di w}{w-f_n(\zeta)}~=~0.
\end{align*}
It follows that $f(\Gamma)$ lies in the unbounded connected component of $\C\setminus C$. By the very construction, the only boundary component of~$f(\Omega)$ possessing this property is its outer boundary.
Recalling that $f(z_0)=0$ and $f'(z_0)\ge f_0'(z_0)$, we see now that $f=f_0$ thanks to the extremal property stated in Remark~\ref{RM_extremal-property}.

\medskip
\noindent\textsc{Proof of (E)}. With (C) having been already proved, relation~\eqref{EQ_shrinking-dist1} is an immediate consequence of the following two facts:
\begin{itemize}
\item[(i)] if $K\subset f_0(\Omega_0\setminus\{\zeta^*\})$ is compact, then $K\subset f_n(\Omega_n)$ for all ${n\in\Natural}$ large enough;\vskip.5ex
\item[(ii)] for any $w\in\partial f(\Omega_0\setminus\{\zeta^*\})$ there exist a sequence $w_n\in\partial f_n(\Omega_n)$ converging to~$w$.
\end{itemize}
We omit the proof of assertions (i) and~(ii) because, up to a few light adjustments, it repeats the standard argument used in the proof of Carath\'eodory's Kernel Convergence Theorem, see e.g. \cite[Theorem~1.8 on p.\,14]{Pommerenke:BoundaryBehaviour}.

To prove relation~\eqref{EQ_shrinking-dist2}, consider the functions $u_n(z):=\omega_{\Omega_n}\big(g_n^{-1}(z),\partial\Omega_0\big)$ defined on the domains ${D_n:=g_n(\Omega_n)}$.
Applying Theorem~\ref{TH_harmonic-measure-via-Log-potential} and Remark~\ref{RM_several components}, we see that
$$
u_n=V_{\mu'_n}-V_{\nu'_n}\quad\text{for each ${n\in\Natural}$,}
$$
where $\mu'_n$ and $\nu'_n$ are positive Borel measures supported on the union of the (non-degenerate) boundary components of $D_n$ corresponding under~$g_n$ to the connected components of $P_0$ and to those of~$P_n$, respectively.
Moreover, $|\mu'_n|=|\nu'_n|$.

In fact, $|\mu'_n|=|\nu'_n|=c_n$ for all~$n\in\Natural$. Indeed, thanks to the conformal invariance of the periods, see Remark~\ref{RM_harm-meas-conformal inv}, this equality follows from relation~\eqref{EQ_several components:charges} applied twice: for the measures $\mu_n$,~$\nu_n$ and for the measures  $\mu'_n$,~$\nu'_n$.

Therefore, on the one hand, by \eqref{EQ_c_n-to-zero}, we have $|\mu'_n|=|\nu'_n|\to0$ as ${n \to +\infty}$. On the other hand, by Claim~1, $u_n(0)=1-\omega_{\Omega_n}(\zeta_n,\Upsilon_n)\to 1$ as ${n \to +\infty}$. It follows that $${\dist\big(0,\supp(\mu'_n+\nu'_n)\big)\to0}\quad\nlim,$$ which is equivalent to~\eqref{EQ_shrinking-dist2}.

\medskip
\noindent\textsc{Proof of (D)}. Recall that by the hypothesis, there exists a compact set ${K\subset\Omega_0\setminus\{\zeta^*\}}$ such that $\zeta_n:=g_n^{-1}(0)\in K$ for all ${n\in\Natural}$. Moreover, the functions $g_n$ are all univalent in their domains and form a normal family in~$\Omega_0\setminus\{\zeta^*\}$. Therefore, it is sufficient to show that $g_n'(\zeta_n)\to0$ as  ${n \to +\infty}$. According to Koebe's $1/4$-Theorem, see e.g. \cite[Theorem~7.8 on p.\,64]{ConwayII},
$$
|g_n'(\zeta_n)|\le\frac{4\,\dist(0,\partial D_n)}{\dist(\zeta_n,\partial\Omega_n)}\quad\text{for all~$n\in\Natural$.}
$$
Thus, (D) follows from~\eqref{EQ_shrinking-dist2}.
\end{proof}

\section{Proof of Theorem~\ref{TH_doubly}}\label{S_doubly}
Let $\Omega\subset\C$ be a doubly connected domain with at least one non-degenerate boundary component. Fix $z\in\Omega$. As we mentioned in the Introduction, the supremum in~\eqref{EQ_main-definition} is achieved for at least one injective holomorphic function $f_*:\Omega\to\UD$ and such a function is said to be \textsl{extremal} in the squeezing function problem. Let $K\subset\UD$ be the bounded connected component of ${\C\setminus f_*(\Omega)}$. Denote by $\varphi$ the conformal mapping of the simply connected domain $D:=f_*(\Omega)\cup K$ onto~$\UD$ normalized by ${\varphi(0)=0}$ and ${\varphi'(0)>0}$. On the one hand, since $f_*$ is extremal, we have
$$
 \dist\big(0,\partial f_*(\Omega)\big)\ge \dist\big(0,\partial\varphi(f_*(\Omega))\big)=\dist\big(0,\varphi(K)\big).
$$
On the other hand, if $\zeta_*\in\UD$ is the point of~$\varphi(K)$ closest to the origin, then by the Schwarz Lemma applied for~$\varphi^{-1}$,
$$
 \dist\big(0,\partial f_*(\Omega)\big)\le|\varphi^{-1}(\zeta_*)|\le|\zeta_*|=\dist\big(0,\varphi(K)\big).
$$
This means that equality takes place in all the above inequalities. Hence, $\varphi^{-1}=\id_\UD$, i.e. $D=\UD$.
Consequently, the outer boundary of $f_*(\Omega)$ coincides with~$\partial\UD$ and
\begin{equation}\label{EQ_S=dist-K}
S_\Omega(z)=\dist(0,\partial K).
\end{equation}
In particular, the statement of Theorem~\ref{TH_doubly} is now obvious for doubly connected domains~$\Omega$ with one degenerate boundary component. Therefore, without loss of generality we may suppose that $\Omega=\A_r$ for some $r\in(0,1)$ and that ${f_*(\partial\UD)=\partial\UD}$. Otherwise we would replace $f_*$ by $f_*\circ g^{-1}$ and $z$ by~$g(z)$, where $g$ is a conformal mapping of $\Omega$ onto~$\A_r$ taking $f_*^{-1}(\partial\UD)$ to~$\partial\UD$.

Then by Corollary~\ref{CR_doubly-conn}\,(A) applied for $f:=f_*$,
$$
\log\frac{1}{|z|}=\log\frac{1}{|f_*^{-1}(0)|}=V_{\mu_*}(0)-V_{\nu_*}(0),
$$
where $\mu_*$ and $\nu_*$ are two probability measures with ${\supp\mu_*=\partial K}$ and ${\supp\nu_*=\partial\UD}$. Clearly, $V_{\nu_*}(0)=0$. Hence, taking into account~\eqref{EQ_S=dist-K}, we have
$$
\log\frac{1}{|z|}=\int\limits_{\partial K}\log\frac{1}{|w|}\di\mu_*(w)\le\log\frac{1}{S_\Omega(z)}.
$$
The equality can occur only if $\partial K$ is contained on the circle of radius~$S_\Omega(z)$ centered at the origin, because $\supp\mu_*$ \textit{coincides} with~$\partial K$.

Taking into account that by Corollary~\ref{CR_doubly-conn}\,(B), ${S_\Omega(z)\ge |z|}$, we conclude that in fact, ${S_\Omega(z)=|z|}$ and that $f_*(\Omega)$ is a circularly slit disk.

To obtain formula~\eqref{EQ_Squeezing-for-the-annulus}, it remains to recall that in case $f_*^{-1}(\partial\UD)$ is $\partial\UD(0,r)$ rather than $\partial\UD$, we have to replace $z$ by $g(z)$, where $g$ is a conformal automorphism of~$\A_r$ permuting the boundary components. The proof is now complete.   \qed

\section{Conformal mapping of an annulus onto a circularly slit disk}\label{S_fx}
Fix $r \in (0,1)$ and consider the annulus $\A_r:=\{z \in \C \, : \, r<|z|<1\}$.
According to Remark~\ref{RM_canonica-map-existence}, for any fixed $x \in (r,1)$ there exist  a unique conformal map~$f_x$ from $\A_r$ into~$\D$ with ${f_x(x)=0}$, ${f_x(\partial \D)=\partial \D}$ and such that $\Gamma_x:=\UD\setminus f_x(\A_r)$ is a circular arc  centred at the origin, symmetric w.r.t.~the real axis, and intersecting the interval $(-1,0)$; see Figure~\ref{FG_fx}.  Note that $f_x$ and $\Gamma_x$  depend also on $r \in (0,1)$, but since $r \in (0,1)$ will be fixed throughout, we suppress the dependence of~$f_x$ on $r$ in our notation.\\
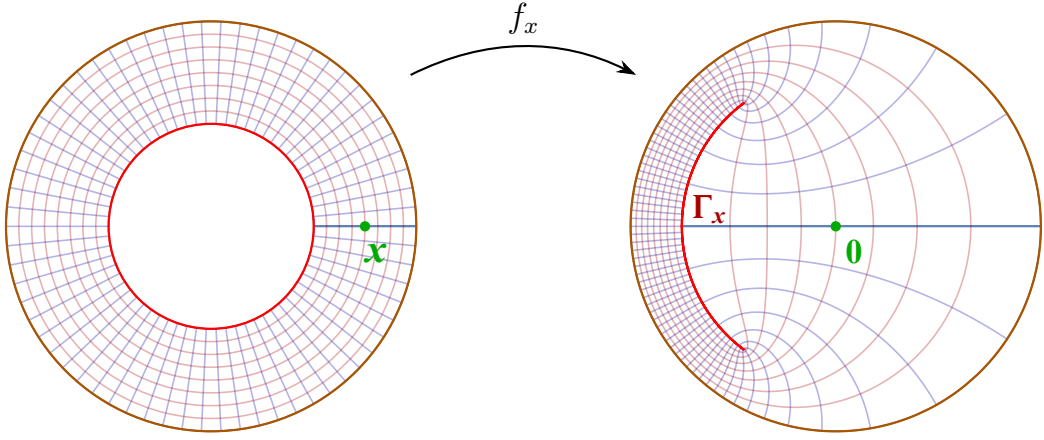
\begin{figure}[h]
\centerline{
\includegraphics[width=6cm]{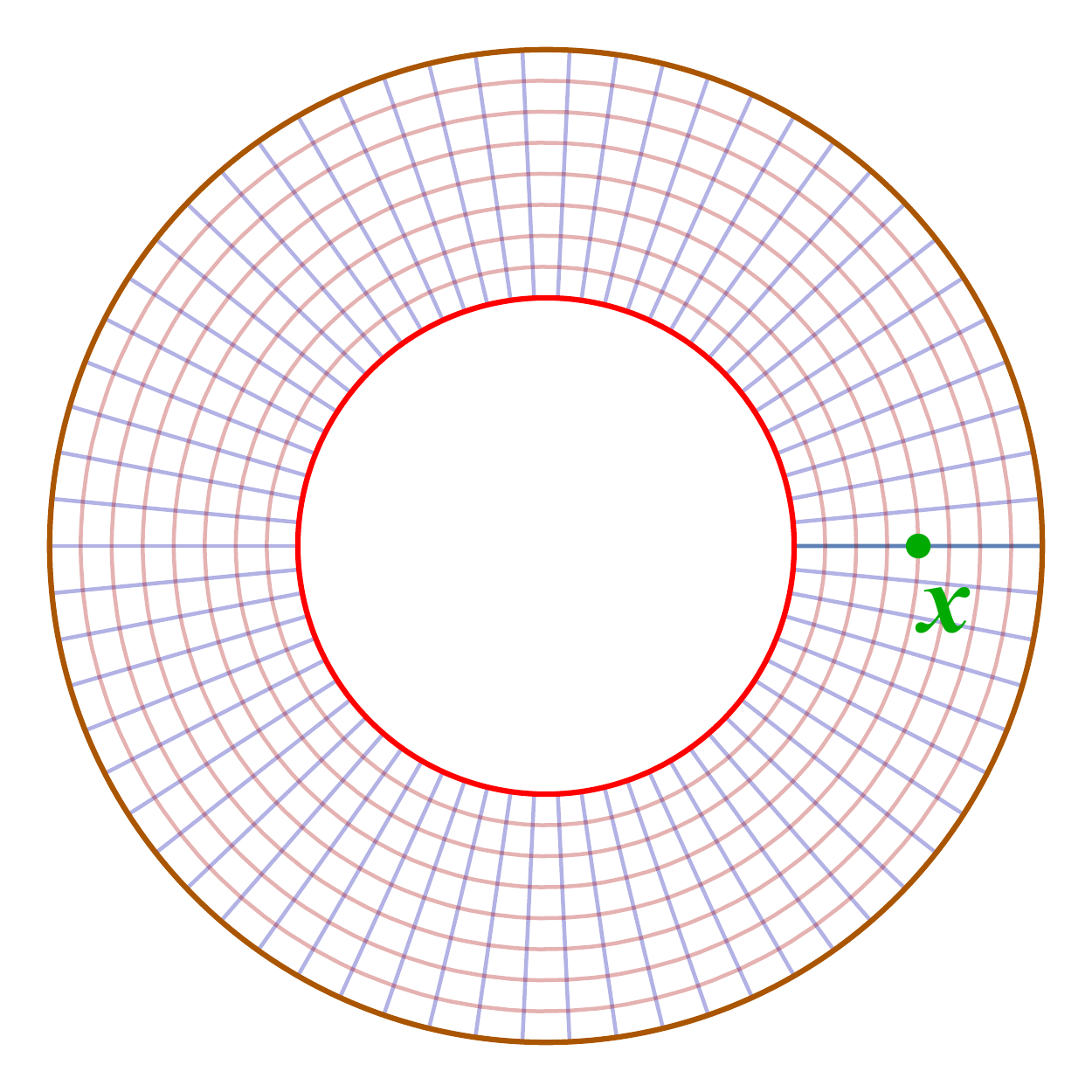}
\hskip-5mm
\raise5cm\hbox{\begin{tikzpicture}
\draw [thick] [-{Stealth[scale=1.25,
          length=6,
          width=4]}] (-1.5,0) .. controls (-0.5,0.5) and (0.5,0.5) .. (1.5,0);
 \draw (0,0.75) node[scale=1.1] {$f_x$};
\end{tikzpicture}}
\hskip-5mm
\includegraphics[width=6cm]{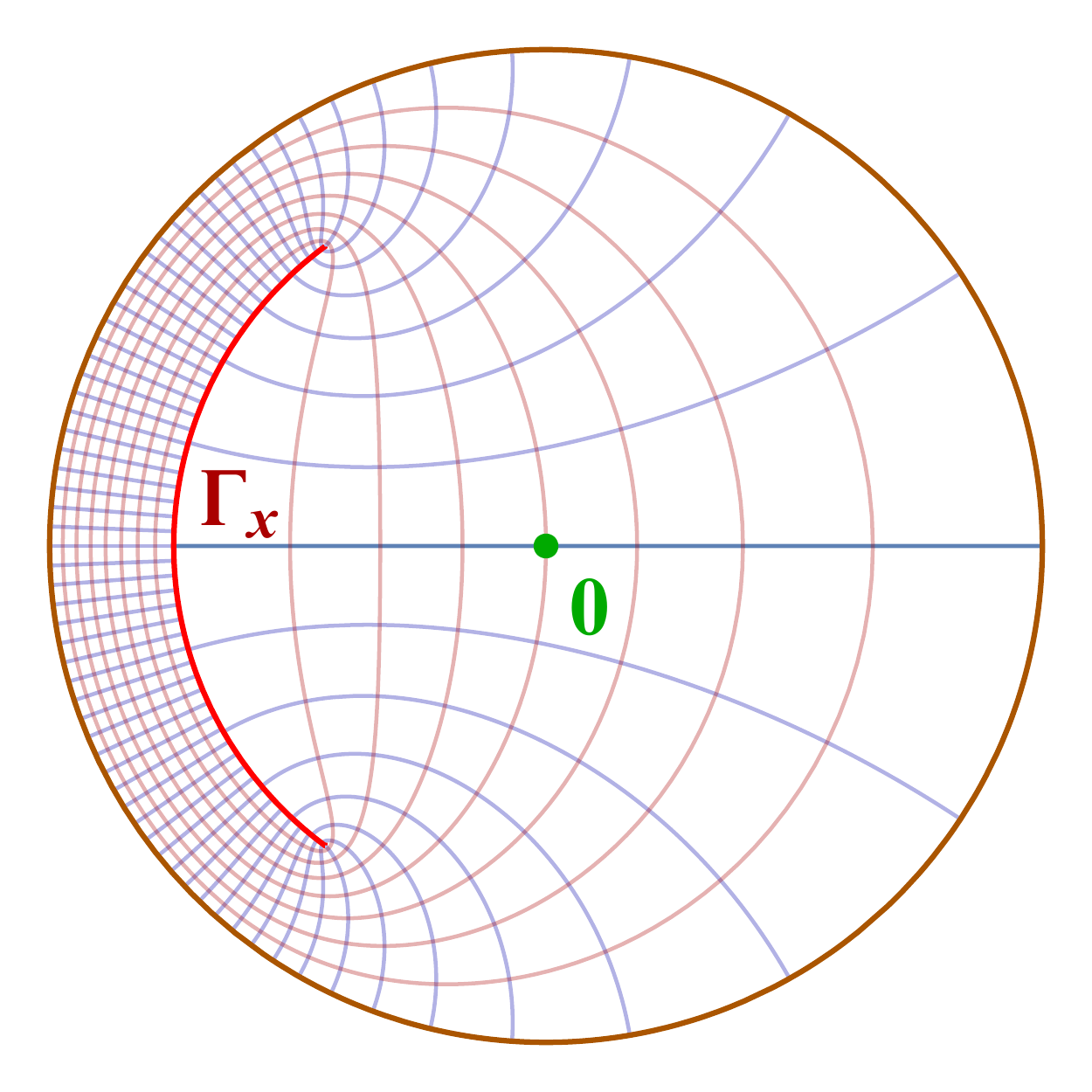}
}
\caption{The map $f_x$ for $r=1/2$ and $x=3/4$}\label{FG_fx}
\end{figure}

In the following lemma we collect some auxiliary statements, which will be used in Section~\ref{S_tripply}  to prove Theorem~\ref{TH_tripply}.

\begin{lemma}[\sc Properties of the conformal map $f_x$] \label{lem:confprop}
Let $r \in (0,1)$ be fixed.\samepage
\begin{itemize}
\item[(A)] For each $x \in (r,1)$ the map $f_x$ has  a holomorphic extension to the annulus $$\A_{r,x}:=\{z \in \C \, : \, r^2/x<|z|<1/x\}$$ and
    $(x,z) \mapsto f_x(z)$ is a $C^{\infty}$--\,function on $\{(z,x)\colon z \in \A_{r,x},\, x \in (r,1) \}$.
\smallskip
\item[(B)] For each $x \in (r,1)$,
      \begin{align} \label{eq:symm}  f_{\alpha}(x) &= -f_x(\alpha)\quad \text{ for all~}~\alpha \in (r,1)  \\
        \intertext{and}
       \label{eq:derivative}  \frac{\di}{\di\alpha} \bigg|_{\alpha=x} f_{\alpha}(x) &= -f_x'(x) < -\frac{1}{1-x^2}.
         \end{align}
\smallskip
\item[(C)] For each $x \in (r,1)$ the circular arc $\Gamma_x$ intersects the real line at the point $-x$.
\smallskip
\item[(D)] If $\Gamma^{+}_x$ denotes the endpoint of the arc $\Gamma_x$ in the upper half-plane, then $x \mapsto \Gamma^{+}_x$ is~a~$C^{\infty}$--\,function in $(r,1)$.
\end{itemize}
\end{lemma}

\newcommand{\vomega}{\hskip.05em\text{\textomega}\hskip.05em}
\begin{remark}[\sc The map $f_x$ via the Schottky\,--\,Klein prime function of $\A_r$] \label{rem:prime}\mbox{}\\
The proof of Lemma~\ref{lem:confprop} can conveniently be based on the Schottky\,--\,Klein prime function $\vomega(z,a)$ of the annulus $\A_r$, which is  defined for all $z,a\in\C^*:=\C \setminus \{0\}$ by

\begin{equation}\label{eq:primedef}
    \vomega(z,a):=(z-a) \prod \limits_{n=1}^{+\infty} \frac{(r^{2n} z-a)(r^{2n} a-z)}{(r^{2n} z-z)(r^{2n} a-a)}= \,
    (z-a) \prod \limits_{n=1}^{+\infty} \frac{(1-r^{2n}z/a)(1-r^{2n}a/z)}{(1-r^{2n})^2},
    \end{equation}
    see \cite[formula (14.57)]{Crowdy}.
    It is known \cite[Section~5.6]{Crowdy} that for each $a \in \A_r$, the function
\begin{equation} \label{eq:primeconformal}
   g_a(z):=\frac{1}{|a|}\frac{\vomega(z,a)}{\vomega(z,1/\overline{a})}
  \end{equation}
 maps $\A_r$ conformally onto a circularly slit disk, with $g_a(a)=0$ and $g_a(\partial \D)=\partial \D$.

   The prime function $\vomega : \C^* \times \C^* \to \Complex$ is holomorphic in both variables and satisfies the following functional identities, see \cite[Section~5.3]{Crowdy}:
\begin{align}
  \vomega(a,z) &= -\vomega(z,a)\, , \label{eq:prime1}\\[.5ex]
  \vomega(\overline{z},\overline{a})&=\overline{\vomega(z,a)}\, ,  \label{eq:prime2} \\[.5ex]
  \vomega(1/z,1/a)&=-\frac{\vomega(z,a)}{za}\,, \label{eq:prime3} \\[.5ex]
  \vomega(r^2 z,a)&= -\frac{a\vomega(z,a)}{z} \, , \label{eq:prime4}
\end{align}
In particular, if $x\in(r,1)$, then \eqref{eq:prime2} yields that $g_x(\overline{z})=\overline{g_x(z)}$. Hence $g_x(\A_r)$ is symmetric w.r.t.~the real axis. A look at \eqref{eq:primedef} reveals that $g_x(r)=x>0$, and it follows that $g_x$ maps $\A_r$ onto $\D$ minus a circular arc symmetric w.r.t.~$\R$ and intersecting the interval~$(0,1)$. This clearly implies $f_x=-g_x$, i.e.
\begin{equation} \label{eq:confprime}
  f_x(z)=-\frac{1}{x} \frac{\vomega(z,x)}{\vomega(z,1/x)} \,
\end{equation}
for any $x \in (0,1)$.
\end{remark}

\begin{proof}[Proof of Lemma~\ref{lem:confprop}]
\textsc{To prove~(A),} notice that according to~\eqref{eq:confprime}, for each $x\in(r,1)$,  $f_x$~extends to a meromorphic function in~$\C^*$. By~\eqref{eq:primedef}, the zeroes of $\vomega(\cdot,1/x)$ are exactly the points ${z_k:=r^{2k}/x}$ with ${k\in\mathbb{Z}}$. Hence, the extension of $f_x$ is holomorphic in~$\A_{r,x}$. Moreover, ${(x,z)\mapsto f_x(z)}$ is of class $C^\infty$ in $\{(x,z)\colon z\in\A_{r,x},\,x\in(r,1)\}$ because the function~$\vomega$ is holomorphic in~${\C^*\times\C^*}$.
\medskip

\newcommand{\eqcomm}[1]{\xlongequal{\raise.5ex\hbox{$\scriptstyle #1$}}}
\noindent\textsc{Proof of~(B)}.  For all $x,\alpha \in (r,1)$, we have
$$ f_{\alpha}(x)=-\frac{1}{\alpha} \frac{\vomega(x,\alpha)}{\vomega(x,1/\alpha)} \eqcomm{(\ref{eq:prime1})} -\frac{1}{\alpha} \frac{\vomega(\alpha,x)}{\vomega(1/\alpha,x)}  \eqcomm{(\ref{eq:prime3})} \frac{1}{x} \frac{\vomega(\alpha,x)}{\vomega(\alpha,1/x)}=-f_x(\alpha)\, .$$
From this and the fact that $\alpha \mapsto f_{\alpha}(x)=-f_x(\alpha)$ is $C^{\infty}$  in $(r,1)$ it follows immediately that
$$ \frac{\di}{\di\alpha} \bigg|_{\alpha=x} f_{\alpha}(x)=-f_x'(x) \, .$$
Using~\eqref{eq:primedef}  and~\eqref{eq:confprime}, we obtain
$$
f'_x(x)=\frac1{1-x^2}\prod_{n=1}^{+\infty}\frac{(1-r^{2n})^2}{(1-r^{2n}x^2)(1-r^{2n}x^{-2})}>\frac1{1-x^2}.
$$
It is worth mentioning that this inequality can be alternatively deduced from the extremal property of conformal mappings onto circularly slit disks, see Remark~\ref{RM_extremal-property}, by comparing~$f_x$ with a suitable automorphism of~$\UD$ restricted to~$\A_r$.

\REM{
Let
$$ L_x(z):=\frac{z+x}{1+x z}\,,\quad z\in\UD.$$
Clearly, $D:=L_x^{-1}(\A_r)$ is a doubly connected domain in~$\UD$ whose outer boundary is $\partial \D$. Moreover,  ${\varphi:=f_x \circ L_x:D\to\UD\setminus\Gamma_x}$ satisfies the hypothesis of Proposition~\ref{PR_Tsuji}. Therefore, ${f_x'(x)L_x'(0)>1}$ or equivalently,
 $$f_x'(x)>\frac{1}{1-x^2},$$ where the strict inequality holds because $D$ is not a circularly slit disk. }

\medskip
\noindent\textsc{Proof of~(C)}. See Corollary~\ref{CR_doubly-conn}\,(B).

\medskip
\noindent\hbox{\textsc{Proof of~(D)}.} Fix~$x\in(r,1)$.  There exists a unique point $r_x \in \partial \D(0,r)$ such that ${f_x(r_x)=\Gamma^+_x}$. Recall that $f_x(\overline z)=\overline{f_x(z)}$ for all ${z\in\A_{x,r}}$. Since $f_x$ is injective and ${f'_x(x)>0}$, it follows that $\Im f_x(z)>0$ for some ${z\in\A_{r,x}}$ if and only if~${\Im z>0}$. In particular,  ${\Im r_x>0}$. Note that the mapping properties of~$f_x$ show that $f_x'(r_x)=0$, but $f_x''(r_x)\neq0$. Moreover, with the help of the Schwarz Reflection Principle we see that $f_x$ is locally injective in $\A_{r,x}\setminus\{r_x,\overline{r_x\!}\,\}$.

Therefore, $r_x$ is the unique solution to the equation ${f_x'(z)=0}$ in ${\{z\in\A_{r,x}\colon \Im z>0\}}$.
Since in view of~(A) the map $(x,z) \mapsto f_x'(z)$ is of class $C^{\infty}$,
the Implicit Function Theorem guarantees that this solution~$r_x$ is also of class $C^{\infty}$
as a function of~$x$.
\end{proof}

\section{Proof of Theorem~\ref{TH_tripply}}\label{S_tripply}

The proof of Theorem~\ref{TH_tripply} is divided into two steps.
We will first show that the analogue of Conjecture~\ref{CNJ} fails for certain once-punctured circularly slit disks of the form
$$ \Omega=\D \setminus (\Gamma \cup \{\zeta^*\}) \, ,  $$
where $\Gamma$ is a non-degenerate circular arc centered at the origin and $\zeta^*$ is a point in $\D \setminus \Gamma$. In other words, formula \eqref{EQ_conjectured} does not hold for such choice of $\Omega$.
 In Section~\ref{SS_non-degenerate} we will see that if the degenerate boundary component~$\{\zeta^*\}$ is replaced by a compact set~$\Upsilon$ consisting of a finite number of non-degenerate circular arcs, then it is still possible to show that formula~\eqref{EQ_conjectured} fails provided that $\Upsilon$ lies within a sufficiently small neighbourhood of~$\zeta^*$.

\subsection{Degenerate case}\label{SS_degenerate}
Fix ${r\in(0,1)}$ and ${x_0\in(\sqrt{r},1)\subset\A_r}$. Using the notation introduced in Section~\ref{S_fx}, for $x \in (r,x_0]$ we let
$$
 \phi_x:=T_x \circ f_x \circ f_{x_0}^{-1}\quad \text{and}\quad {\Gamma(x):=T_x(\Gamma_x)}, \quad \text{ where~}~T_x(z):=\frac{z-f_x(x_0)}{1-f_{x}(x_0) z}.
$$
Then $\phi_x$ maps the circularly slit disk $\Omega_0:=f_{x_0}(\A_r)=\D \setminus \Gamma_{x_0}$ conformally onto
${\D \setminus \Gamma(x)}$, with ${\phi_x(0)=0}$ and ${\phi_x(\partial \D)=\partial \D}$. The set
$\Gamma(x)$ is a closed subarc of the circle $T_x\big(\partial \D(0,x)\big)$ symmetric w.r.t. the real line.

We will see that there exists $x^*\in(r,x_0)$ and~${\zeta^*\in(-x_0,0)}$ such that
\begin{equation}\label{EQ_main-for-degenerate-case}
\dist\big(0,\partial(\Omega_0\setminus\{\zeta^*\})\big) <\dist\big(0,\partial\phi_{x^*}(\Omega_0\setminus\{\zeta^*\})\big).
\end{equation}
Using this fact, it is easy to show that the formula \eqref{EQ_conjectured} of Conjecture~\ref{CNJ} does not hold for~$\Omega:=\Omega_0\setminus\{\zeta^*\}$ and ${z:=0}$.
The proof of~\eqref{EQ_main-for-degenerate-case} is based on the following lemma.
\begin{lemma}\label{LM_phi_x}
Let $0<r<x_0<1$ be fixed. Then
for any $\,\varepsilon\in(0,x_0)$ there exists $x^* \in (r,x_0)$ and $\delta \in (0,\varepsilon)$ such that\samepage
  \begin{itemize}
  \item[(i)] $\phi_{x^*}(\xi)<\xi$ for all $\xi \in (-x_0,-x_0+\delta)$;\vskip.5ex
  \item[(ii)] $\dist(0,\Gamma(x^*))>x_0-\delta$.
  \end{itemize}
\end{lemma}
We postpone the proof to the end of this section.
Choose any $x_0\in(\sqrt r,1)$ and, taking Lemma~\ref{LM_phi_x} for granted, apply it with ${\varepsilon:=x_0-r/x_0}$. By assertion~(ii), there exists ${\zeta^*\in(-x_0,-x_0+\delta)}$ such that $|\zeta^*|<\dist(0,\Gamma(x^*))$.
In combination with~(i), this leads to
\begin{equation}\label{EQ->|zeta-star|}
 \dist\big(0,\partial\phi_{x^*}(\Omega_0\setminus\{\zeta^*\})\big)=
 \min\big\{|\phi_{x^*}(\zeta^*)|,\dist(0,\Gamma(x^*))\big\}>|\zeta^*| .
\end{equation}
Then~\eqref{EQ_main-for-degenerate-case} holds because, trivially, $\dist\big(0,\partial(\Omega_0\setminus\{\zeta^*\})\big)<|\zeta^*|$.
\medskip

Now, if  formula \eqref{EQ_conjectured} in  Conjecture~\ref{CNJ}  would hold for $\Omega:=\Omega_0\setminus\{\zeta^*\}$ and ${z:=0}$, then
$$
\dist\big(0,\partial\phi_{x^*}(\Omega)\big)\le \dist\big(0,\partial f_{0,0}(\Omega)\big)\quad\text{or} \quad
\dist\big(0,\partial\phi_{x^*}(\Omega)\big)\le \dist\big(0,\partial f_{0,1}(\Omega)\big),
$$
with $f_{0,0}=\id_\Omega$ and $f_{0,1}=f_{r/x_0}\circ\big(z\mapsto r/z\big)\circ f_{x_0}^{-1}$. However, the first inequality is the opposite of~\eqref{EQ_main-for-degenerate-case}, while the latter one cannot hold because of the choice of~$x_0$ and $\varepsilon>0$; indeed,
\begin{equation}\label{EQ-<|zeta-star|}
\dist\big(0,\partial f_{0,1}(\Omega)\big)\le \dist\big(0,\partial f_{r/x_0}(\A_r)\big)=\frac{r}{x_0}=x_0-\varepsilon\le x_0-\delta<|\zeta^*|.
\end{equation}

To facilitate the proof of Lemma~\ref{LM_phi_x}, we will first establish the following statement.
\begin{lemma}\label{LM_dist-Lip}
The map $x \mapsto \dist\big(0,\Gamma(x)\big)$ is locally Lipschitz on $(r,x_0]$.
\end{lemma}
\begin{proof}
We know from Lemma~\ref{lem:confprop}  that  $x \mapsto \Gamma_x^+$, where  $\Gamma_x^+$ denotes the endpoint of the circular arc $\Gamma_x$ in the upper half-plane, as well as $x \mapsto f_x(x_0)$ are $C^{\infty}$--\,functions on~$(r,1)$. Hence
$$
 x \mapsto T_x(\Gamma^+_x)= \frac{\Gamma_x^+-f_x(x_0)}{1-f_x(x_0) \Gamma^+_x}
$$ is also $C^{\infty}$, in particular locally Lipschitz.

Note that for $x\in(r,x_0)$ the center of the circular arc $\Gamma(x)$ is not the origin; in fact, it belongs to~$(-1,0)$ because $f_x(x_0)>0$. It follows that $$(r,x_0]\ni x\mapsto \dist(0,\Gamma(x))=|T_x(\Gamma^+_x)|$$ is indeed locally Lipschitz.
\end{proof}

Now we are ready to \textit{prove Lemma~\ref{LM_phi_x}}.

\Step1 We first show that
\begin{equation}\label{EQ_derivative_new}
 \phi_x'(\xi)\to1\quad\text{uniformly on~}~[-x_0,0]\quad \text{as $~x\to x_0$}.
\end{equation}
Indeed, $\phi_x\circ f_{x_0}=T_x\circ f_x$ and hence
$$
\phi_x'(\xi)\,f_{x_0}'(f_{x_0}^{-1}(\xi))=\big(T_x\circ f_x\big)'(f_{x_0}^{-1}(\xi))\quad\text{for all~}~\xi\in[-x_0,0].
$$
By Lemma~\ref{lem:confprop}\,(A), the map $(\xi,x)\mapsto \big(T_x\circ f_x\big)'(\xi)$ is of class~$C^\infty$ on the set $$\big\{(\xi,x)\colon \xi\in(r^2/x,1/x),\,x\in(r,1)\big\}.$$
Since $T_{x_0}=\id_\UD$, it follows that, as $x\to x_0$, $\big(T_x\circ f_x\big)'\to f'_{x_0}$ uniformly on
 $[r,x_0]={f_{x_0}^{-1}\big([-x_0,0]\big)}$.
It only remains to observe that $f'_{x_0}$ does not vanish on~$[r,x_0]$.

\Step2 Let $x \in (r,1)$ and denote by  $$q(x):=\phi_x(-x_0)=T_x(-x)=-\frac{x+f_x(x_0)}{1+f_x(x_0)x} $$
the point of intersection of $\Gamma(x)$ with the real line.  Lemma~\ref{lem:confprop}  shows that $q : (r,1) \to \R$ is differentiable with $$q'(x_0)=-\big(1-(1-x_0^2) f_{x_0}'(x_0)\big) >0\quad\text{and}\quad q(x_0)=-x_0.$$
Hence, there is $x_1 \in (r,x_0)$ such that
$ \phi_x(-x_0)=q(x)<-x_0$ for all $x \in (x_1,x_0)$. In particular, for any $x \in (x_1,x_0)$ we can define
$$ \delta(x):=\sup \big\{ \alpha \in (0,x_0)\,\colon \phi_x(\xi)<\xi ~\text{~for all~}~ \xi \in [-x_0,-x_0+\alpha) \big\}>0 \, .$$
We claim that
\begin{equation}
\label{EQ_delta-asympt_new}
\lim \limits_{x \nearrow x_0} \frac{\delta(x)}{|x-x_0|}=+\infty \, .
\end{equation}
In order to prove this, fix an arbitrary $\varepsilon_1>0$. Thanks to~\eqref{EQ_derivative_new}, we can find $x_2 \in (x_1,x_0)$ such that
$\phi_x'< 1+\varepsilon_1$ on $[-x_0,0]$ for all $x \in (x_2,x_0)$. Then, for any fixed $x\in(x_2,x_0)$ and any $\xi \in [-x_0,0]$ satisfying $\xi \le -x_0-\big(\phi_x(-x_0)+x_0\big)/\varepsilon_1,$
we have
\begin{eqnarray*}
  \phi_x(\xi)&=& \phi_x(-x_0) + {\int\limits_{-x_0}^{\xi}\! \phi_x'(s) \di s}\\ & <& \phi_x(-x_0)+ (1+\varepsilon_1) (\xi+x_0) \\[.5ex] &\le &   \phi_x(-x_0)+(\xi+x_0)-(\phi_x(-x_0)+x_0)\\[.5ex] &= &\xi \, .
\end{eqnarray*}
This shows that
$$ \delta(x) \ge -\frac{\phi_x(-x_0)+x_0}{\varepsilon_1} \quad \text{ for all } x \in (x_2,x_0) \, ,$$
and therefore, recalling that $x_0=-\phi_{x_0}(-x_0)$, we have
$$ \liminf \limits_{x \nearrow x_0} \frac{\delta(x)}{x_0-x} \ge \frac{1}{\varepsilon_1} \liminf \limits_{x \nearrow x_0}
\frac{\phi_{x_0}(-x_0)-\phi_x(-x_0)}{x_0-x} =\frac{q'(x_0)}{\varepsilon_1} \, . $$

Since $q'(x_0)>0$ as we have observed above and since $\varepsilon_1>0$ can be chosen arbitrarily small, the latter inequality implies (\ref{EQ_delta-asympt_new}).

\Step3 Since by Lemma~\ref{LM_dist-Lip},  $x \mapsto \dist(0,\Gamma(x))$ is locally Lipschitz and  ${\dist(0,\Gamma(x_0))=x_0}$, there exist a constant ${M>0}$ and a point ${x_3 \in (r,x_0)}$ such that
$$
 \dist(0,\Gamma(x))>x_0-M(x_0-x)\quad \text{for all~}~ {x \in (x_3,x_0)}.
$$
In view of (\ref{EQ_delta-asympt_new}), there exists $x^*\in (x_3,x_0)$ such that
$$ M(x_0-x^*) \le \min \{\varepsilon,\delta(x^*)\}=:\delta \, , $$
and with this choice of $\delta$ both conditions  (i) and (ii) hold by construction. \qed

\subsection{Non-degenerate case}\label{SS_non-degenerate}
As we have seen in the previous section, for any $r\in(0,1)$ and any $x_0\in(\sqrt r,1)$, there exists $\zeta^*\in f_{x_0}(\A_r)$ such that formula~\eqref{EQ_conjectured} in Conjecture~\ref{CNJ} fails for $\Omega:=f_{x_0}(\A_r)\setminus\{\zeta^*\}$ and ${z:=0}$.

We will now use Theorem~\ref{TH_shrinking-conv} to show that for any $m\ge 3$, there exist $m-2$ pairwise disjoint closed non-degenerate arcs $\Gamma_2,\ldots,\Gamma_{m-1}\subset f_{x_0}(\A_r)$ located on circles centered at the origin, such that formula~\eqref{EQ_conjectured} fails also for
$$
\Omega:=f_{x_0}(\A_r)\,\big\backslash\bigcup_{j=2}^{m-1}\Gamma_j=\UD\,\big\backslash\bigcup_{j=1}^{m-1}\Gamma_j \quad\text{and}\quad{z:=0},
$$
where we set ${\Gamma_1:=\Gamma_{x_0}}$ and $\Gamma_0:=\partial\UD$.

To this end we consider $m-2$ sequences of arcs $(\Gamma^j_n)_{n\in\Natural}$, $j=2,\ldots,m-1$, such that
for each $n\in\Natural$ fixed, $\Gamma^2_n,\ldots,\Gamma^{m-1}_n\subset f_{x_0}(\A_r)$ are pairwise disjoint closed non-degenerate arcs located on circles centered at the origin and $\Upsilon_n:=\bigcup_{j=2}^{m-1}\Gamma^j_n\subset\UD(\zeta^*,1/n)$.
Denote ${\Omega_0:=f_{x_0}(\A_r)}$. Further, for each ${n\in\Natural}$, let ${\Omega_n:=\Omega_0\setminus\Upsilon_n}$
and denote by $f^{z,j}_n$, with ${z\in\Omega_n}$ and ${j=0,\ldots,m-1}$, the conformal mappings of $\Omega_n$ onto circularly slit disks, as introduced in Remark~\ref{RM_canonica-map-existence}. The slit disk mappings of the domain~$\Omega_0$ will be denoted, as in Section~\ref{SS_degenerate}, by $f_{z,0}$ and $f_{z,1}$.

Clearly, $f^{0,0}_n=\id_{\Omega_n}$. Taking into account that $\phi_{x^*}$ is holomorphic at~$\zeta^*$, it is easy to see that, as~${n\to+\infty}$,
$$
 \dist\big(0,\partial\Omega_n\big)\to \dist\big(0,\partial\,(\Omega_0\setminus\{\zeta^*\})\big)~\text{~and~}~ \dist\big(0,\partial\phi_{x^*}(\Omega_n)\big) \to \dist\big(0,\partial\phi_{x^*}(\Omega_0\setminus\{\zeta^*\}\big).
$$
Hence in view of~\eqref{EQ_main-for-degenerate-case}, for all $n\in\Natural$ large enough we have
\begin{equation}\label{EQ_0}
\dist\big(0,\partial f^{0,0}_n(\Omega_n)\big)=\dist\big(0,\partial\Omega_n\big)<\dist\big(0,\partial\phi_{x^*}(\Omega_n)\big).
\end{equation}

By relation~\eqref{EQ_shrinking-dist1} in Theorem~\ref{TH_shrinking-conv} applied with $\Gamma:=\Gamma_1=\Gamma_{x_0}$, $\dist\big(0,\partial f^{0,1}_n(\Omega_n)\big)\to \dist\big(\partial f_{0,1}(\Omega_0\setminus\{\zeta^*\})\big)$ as ${n\to+\infty}$. Therefore, in accordance with \eqref{EQ->|zeta-star|} and~\eqref{EQ-<|zeta-star|}, for all $n\in\Natural$ large enough we also have
\begin{equation}\label{EQ_1}
\dist\big(0,\partial f^{0,1}_n(\Omega_n)\big)<\dist\big(0,\partial\phi_{x^*}(\Omega_n)\big).
\end{equation}

Finally, for each $j=2,\ldots,m-1$, thanks to relation~\eqref{EQ_shrinking-dist2} in Theorem~\ref{TH_shrinking-conv},
\begin{equation}\label{EQ_j}
\dist\big(0,\partial f^{0,j}_n(\Omega_n)\big)<\dist\big(0,\partial\phi_{x^*}(\Omega_n)\big)
\end{equation}
provided $n\in\Natural$ is large enough.
From~\eqref{EQ_0}, \eqref{EQ_1}, and \eqref{EQ_j}, it immediately follows that Conjecture~\ref{CNJ} fails for $\Omega:=\Omega_n$ and ${z:=0}$ if $n\in\Natural$ is sufficiently large. This completes the proof of Theorem~\ref{TH_tripply}.\qed

\end{document}